\documentclass[11pt]{article}
\usepackage[margin=1in]{geometry}
\usepackage{amsmath,amsthm,amssymb,mathtools,bm}
\usepackage[T1]{fontenc}
\usepackage{lmodern}
\usepackage{microtype}
\usepackage[round,authoryear]{natbib}
\usepackage{hyperref}
\usepackage{enumitem}
\usepackage{mathrsfs}

\hypersetup{colorlinks=true,linkcolor=blue,citecolor=blue,urlcolor=blue}

\newtheorem{theorem}{Theorem}[section]
\newtheorem{proposition}[theorem]{Proposition}
\newtheorem{lemma}[theorem]{Lemma}
\newtheorem{corollary}[theorem]{Corollary}
\theoremstyle{definition}
\newtheorem{definition}[theorem]{Definition}
\newtheorem{remark}[theorem]{Remark}

\title{Characterisations of Kullback--Leibler approximation by finite Gaussian mixtures}
\author{Hien Duy Nguyen\\
\small Department of Mathematics and Physical Sciences, La Trobe University, Australia\\
\small Institute of Mathematics for Industry, Kyushu University, Japan\\
\small Email: \texttt{h.nguyen5@latrobe.edu.au}}
\date{}

\begin{document}
\maketitle

\begin{abstract}
We study the Kullback--Leibler (KL) divergence approximation theory of Gaussian mixture models (GMMs) by isolating an abstract mechanism behind several necessary-and-sufficient statements. The necessity direction is universal: if a density is approximable in KL divergence by finite GMMs, then it must have finite second moment. The sufficient direction is reduced to the construction of approximating GMMs whose likelihood ratios converge pointwise and whose finite log-ratios form a uniformly integrable family. We verify this mechanism on a finite log-moment class of continuous strictly positive target densities, from which bounded, $\mathcal L^p$ $(p>1)$, and Orlicz-dominated subfamilies follow immediately. We also show that a countable-scale support-aware target density class, which allows zero density regions, satisfies the same equivalence. Finally, we give counterexamples showing that the countable-scale class strictly extends the fixed-scale class, that the finite log-moment and countable-scale support-aware classes do not contain one another, and that their union is not exhaustive.
\end{abstract}

\medskip
\noindent\textbf{Keywords.} Gaussian kernel; finite mixture models; Kullback--Leibler divergence; approximation theory; equivalence characterizations 

\section{Introduction}\label{sec:intro}

Gaussian mixture models (GMMs) occupy a central position in probability, statistics, and allied areas of applied mathematics. A Gaussian mixture density on $\mathbb{R}^d$ is a probability density of the form
\[
 g(x)=\sum_{j=1}^m \pi_j\,\varphi(x;\mu_j,\Sigma_j),
\qquad m\in\mathbb N,
\quad \pi_j\ge 0,
\quad \sum_{j=1}^m \pi_j=1,
\]
where $\mu_j\in\mathbb{R}^d$, $\Sigma_j\succ0$, and $\varphi(\cdot;\mu_j,\Sigma_j)$ denotes the Gaussian density with mean $\mu_j$ and covariance matrix $\Sigma_j$. Gaussian mixing distributions and finite GMMs are widely used because they simultaneously provide interpretable statistical models for heterogeneous populations and flexible approximation devices for unknown probability densities; standard references include \citet{McLachlanPeel2000} and \citet{Chen2023}. The flexibility of mixture models, and in particular the frequently invoked claim that Gaussian mixtures can approximate arbitrary densities, is reviewed in \citet{NguyenMcLachlan2019}, who collect a number of statements of this folk theorem.

For the remainder of the manuscript, $\lambda$ denotes Lebesgue measure on $\mathbb{R}^d$. If $f$ is the target density, we write
\[
F(A)=\int_A f(x)\,\mathrm{d}x,
\qquad A\in\mathcal B(\mathbb{R}^d),
\]
for the corresponding probability measure. For densities $f$ and $g$, we write
\[
\mathrm{KL}(f\|g)=\int_{\mathbb{R}^d} f(x)\log\!\left(\frac{f(x)}{g(x)}\right)\,\mathrm{d}x.
\]
For $p\in[1,\infty]$, $\mathcal L^p$ denotes the usual Lebesgue space with respect to $\lambda$.

Approximation results for Gaussian kernels have a long history. At a classical level, the Tauberian theory of \citet{Wiener1932} already identifies density properties of translates and convolutions of Gaussian kernels. In the approximation-theory and neural-network literatures, the universal approximation properties of radial basis function networks, including Gaussian kernels, are proved, for example, in \citet{ParkSandberg1991} and \citet{ParkSandberg1993}; see also \citet[Chapter~20]{CheneyLight2009}. When one additionally imposes positivity and unit-mass constraints so that the approximants themselves remain densities, the problem becomes more subtle. For example, \citet{Bacharoglou2010}, \citet{NguyenEtAl2020}, and \citet{NguyenTTEtAl2022} establish qualitative or $\mathcal L^p$-type approximation results for finite location-scale mixtures and related convex mixture classes. 

From the statistical point of view, the Kullback--Leibler (KL) divergence is the natural approximation criterion because it is the population target of maximum likelihood under misspecification; see \citet{White1982}. This is particularly relevant for mixtures, since finite mixture models are typically fitted by maximum likelihood or penalized likelihood procedures in practice; see, for example, \citet{McLachlanPeel2000} and \citet{Chen2023}. In particular many theoretical frameworks for finite-mixture approximation and estimation are formulated around the minimum KL divergence, or around criteria that are directly implied by it. Examples include the approximation and estimation bounds of \citet{ZeeviMeir1997}, \citet{LiBarron2000}, and \citet{RakhlinPanchenkoMukherjee2005}; the Bayesian and penalized-mixture approximation programs of \citet{KruijerRousseauvdV2010}, \citet{MaugisMichel2011}, and \citet{MaugisRabusseauMichel2013}; the minimax theory for Gaussian location mixtures of \citet{KimGuntuboyina2022}; and the recent KL-stability analysis for Gaussian non-parametric maximum likelihood estimators by \citet{GhoshGuntuboyinaMukherjeeTran2026}.

Unfortunately, the aforementioned $\mathcal L^p$ approximation results do not apply to KL approximation. Indeed, KL controls $\mathcal L^1$ through Pinsker's inequality, but it is only bounded above by a $\mathcal L^2$-distance under strong positivity assumptions on both the target and the approximant; see Lemma~3 of \citet{ZeeviMeir1997} and the discussion in \citet{NguyenMcLachlan2019}. An important KL approximation theorem was obtained by \citet[Theorem~3]{NoretsPelenis2012}. Roughly speaking, \cite{NoretsPelenis2012} identify a fixed-scale support-regularity condition, formulated in terms of local log-oscillation control and a cube-thickness requirement near the support boundary, under which finite second moment implies KL approximation by finite Gaussian mixtures. We record the precise fixed-scale support condition later in Definition~\ref{def:Sfix}. Immediately after their theorem, \cite{NoretsPelenis2012} note that the finite second-moment assumption is the strongest one and that the proof suggests weakening it only by enlarging the component family beyond Gaussian densities. The starting point of the present manuscript is to show that this obstruction is genuine for finite GMMs. That is, finite second-moment existence is not only sufficient on suitable classes, but universally necessary for KL approximability by finite GMMs.

The contribution of this manuscript is fourfold. First, we prove a universal necessity theorem: if a density admits finite KL divergence to some finite GMM, and hence a fortiori if it is approximable in KL by a sequence $(g_m)_{m\ge1}$ of such mixtures, then its second moment must be finite. Second, we isolate an abstract sufficiency argument. The key observation is that if one can construct finite GMMs $(g_m)_{m\ge1}$ such that $f/g_m\to1$ almost surely under the target law and the corresponding finite log-ratios $\left(\log_+(f/g_m)\right)_{m\ge1}$ are uniformly integrable, then $\mathrm{KL}(f\|g_m)\to0$. Here and throughout, $\log_+(t)=\max\{\log t,0\}$ for $t>0$, with the convention $\log_+0=0$. Combined with the universal necessity theorem, this yields a meta necessary and sufficient condition on any class for which those two properties can be verified.

Third, we verify this abstract mechanism by two distinct approaches. The first is a finite log-moment approach: continuity, strict positivity, and finiteness of $\int f\log_+ f\,\mathrm{d}x$ are enough to verify the uniform-integrability step, and this approach immediately covers bounded, $\mathcal L^p$ $(p>1)$, and more general Orlicz-dominated subfamilies. It also contains the target-density classes used in the logarithmic H\"older mixture literature of \citet{KruijerRousseauvdV2010} and \citet{MaugisRabusseauMichel2013}; see Remark~\ref{rem:rateclasses}. The second approach is support-aware and follows the local lower-envelope structure used by \citet{NoretsPelenis2012}. We first restate their fixed-scale target class and then introduce a countable-scale extension (cf. Definition~\ref{def:Scnt}) obtained by partitioning the support into pieces whose normalized restrictions each satisfy that fixed-scale condition, with summability strong enough to recover the same KL equivalence with the second moment.

Fourth, we give counterexamples clarifying the relation between the two approaches. The countable-scale support-aware class is a genuine extension of the fixed-scale class, and it contains densities with support gaps and zeros that lie outside the present finite log-moment class of continuous strictly positive densities. We also exhibit a finite log-moment density that fails the countable-scale support-aware condition. Together with the strict-extension example, this shows that the finite log-moment and countable-scale support-aware classes do not contain one another. Finally, we exhibit a KL-approximable density outside both sufficient classes, showing that the overall sufficient theory is not exhaustive.

The remainder of the manuscript is organised as follows. Section~\ref{sec:notation} collects notation and preliminary definitions. Section~\ref{sec:meta} proves the necessity of the second moment and states the abstract meta theorem. Section~\ref{sec:classes} verifies the abstract hypothesis on several concrete classes, including the finite log-moment class and the fixed-scale and countable-scale support-aware classes. Section~\ref{sec:counterexamples} presents the strict-extension and non-exhaustiveness counterexamples. Section~\ref{sec:conclusion} closes with a brief conclusion. Useful auxiliary results are recorded in Appendix~\ref{app:external}.

\section{Notation and technical preliminaries}\label{sec:notation}

We retain the notation introduced in Section~\ref{sec:intro} throughout. All densities are understood with respect to Lebesgue measure $\lambda$ on $\mathbb{R}^d$; when the integration variable is explicit, we write this measure simply as $\mathrm{d}x$, $\mathrm{d}y$, and so on. If $f$ is a density, then $F$ denotes the probability measure with density $f$, so that
\[
F(A)=\int_A f(x)\,\mathrm{d}x
\]
for measurable $A\subseteq\mathbb{R}^d$. Throughout the manuscript, we write
\[
\mathrm{supp}(f)=\{x\in\mathbb{R}^d:f(x)>0\}
\]
for the support convention used in the support-aware arguments below. For $R>0$,
\[
B_R=\{x\in\mathbb{R}^d:\left\lVert x\right\rVert\le R\},
\qquad
B_R^c=\mathbb{R}^d\setminus B_R,
\]
where $\left\lVert\cdot\right\rVert$ is the Euclidean norm, and $I_d$ denotes the $d\times d$ identity matrix. We write $\sqcup$ for disjoint union, $\mathbf{1}_A$ for the indicator of a set $A$, and
\[
\operatorname{dist}(x,A)=\inf_{a\in A}\left\lVert x-a\right\rVert
\]
for the Euclidean distance from a point $x$ to a set $A$. Unless stated otherwise, comparison cubes such as $C(r,y)$ and $C_\ell(y)$ are closed and axis-aligned. Partition cubes may use the usual half-open endpoint convention needed to obtain disjoint unions. A cube of side length $r$ with a vertex at $y$ means an axis-aligned closed cube of side length $r$ having $y$ as one of its vertices. A family of cubes of common side length is called adjacent if the cubes have pairwise disjoint interiors and the union of their closures is connected. We retain the notation $\log_+$ from Section~\ref{sec:intro}, and we use the continuous extension $u\log\!\left(1/u\right)=0$ at $u=0$. Finally,
\[
\mathcal M=
\left\{
x\mapsto \sum_{j=1}^m \pi_j\,\varphi(x;\mu_j,\Sigma_j):
m\in\mathbb N,\ 
\pi_j\ge0,\ 
\sum_{j=1}^m \pi_j=1,\ 
\mu_j\in\mathbb{R}^d,\ 
\Sigma_j\succ0
\right\}
\]
where
\[
\varphi(x;\mu,\Sigma)
=
(2\pi)^{-d/2}|\Sigma|^{-1/2}
\exp\!\left(-\frac12 (x-\mu)^\top \Sigma^{-1}(x-\mu)\right),
\qquad x\in\mathbb{R}^d,
\]
for $\mu\in\mathbb{R}^d$ and $\Sigma\succ0$. Thus $\mathcal M$ denotes the class of all finite GMM densities on $\mathbb{R}^d$.

\section{Necessity of the second moment and a meta theorem}\label{sec:meta}

This section isolates the universal part of the argument. The key point is that KL convergence follows once the finite log-ratios are uniformly integrable and the likelihood ratios converge pointwise, whereas finite second moment is necessary independently of any class-specific approximation scheme.

\subsection{Uniform integrability of the finite log-ratios}

We begin by fixing the uniform-integrability notion used throughout the manuscript.

\begin{definition}[Uniform integrability]
Let $(X_m)_{m\ge1}$ be a family of nonnegative random variables on $(\mathbb{R}^d,\mathcal B(\mathbb{R}^d),F)$. We say that $(X_m)_{m\ge1}$ is uniformly integrable under $F$ if
\[
\lim_{M\to\infty}\sup_{m\ge1}\int_{\{X_m>M\}} X_m\,\mathrm{d}F =0.
\]
\end{definition}

The next proposition is the basic mechanism behind the sufficient direction.

\begin{proposition}
\label{prop:UIKL}
Let $(g_m)_{m\ge1}$ be densities on $\mathbb{R}^d$ such that:
\begin{enumerate}[label=\textnormal{(\roman*)}]
    \item $g_m(x)>0$ for $F$-almost every $x$;
    \item $f(x)/g_m(x)\to1$ for $F$-almost every $x$;
    \item the family $\{L_m\}_{m\ge1}$, where
    \[
    L_m=\log_+\!\left(\frac{f}{g_m}\right),
    \]
    is uniformly integrable under $F$.
\end{enumerate}
Then
\[
\mathrm{KL}(f\|g_m)\to0.
\]
\end{proposition}

\begin{proof}
By Gibbs' inequality, $\mathrm{KL}(f\|g_m)\ge0$ for every $m$. Since the family $\{L_m\}_{m\ge1}$ is uniformly integrable under $F$, each $L_m$ is integrable. Moreover,
\[
\log\!\left(\frac{f}{g_m}\right)\le \log_+\!\left(\frac{f}{g_m}\right)=L_m
\]
pointwise, and therefore
\begin{equation}
\label{eq:KLleL}
0\le \mathrm{KL}(f\|g_m)
=\int \log\!\left(\frac{f}{g_m}\right)\,\mathrm{d}F
\le \int L_m\,\mathrm{d}F.
\end{equation}
By assumption, $L_m(x)\to0$ for $F$-almost every $x$.

It therefore remains to show that $\int L_m\,\mathrm{d}F\to0$. Let $\varepsilon>0$. Uniform integrability yields an $M<\infty$ such that
\[
\sup_{m\ge1}\int_{\{L_m>M\}}L_m\,\mathrm{d}F<\varepsilon.
\]
Then
\[
\int L_m\,\mathrm{d}F
=\int_{\{L_m\le M\}}L_m\,\mathrm{d}F+\int_{\{L_m>M\}}L_m\,\mathrm{d}F.
\]
The second term is bounded by $\varepsilon$ uniformly in $m$. On $\{L_m\le M\}$ we have $0\le L_m\le M$, and $L_m\to0$ $F$-almost everywhere, so dominated convergence gives
\[
\int_{\{L_m\le M\}}L_m\,\mathrm{d}F\to0.
\]
Hence $\limsup_m\int L_m\,\mathrm{d}F\le\varepsilon$. Since $\varepsilon>0$ was arbitrary, $\int L_m\,\mathrm{d}F\to0$. Returning to \eqref{eq:KLleL}, we conclude that $\mathrm{KL}(f\|g_m)\to0$.
\end{proof}

\subsection{Finite second moment is always necessary}

The necessity statement rests on the observation that every finite GMM has a nontrivial quadratic exponential moment.

\begin{lemma}
\label{lem:gauss-exp}
Let $g\in\mathcal M$. Then there exists $t_0>0$ such that
\[
\int_{\mathbb{R}^d} e^{t_0\left\lVert x \right\rVert^2} g(x)\,\mathrm{d}x<\infty.
\]
\end{lemma}

\begin{proof}
Write
\[
g(x)=\sum_{j=1}^m \pi_j\varphi(x;\mu_j,\Sigma_j).
\]
For each $j$, let $Z_j\sim \mathrm{N}(\mu_j,\Sigma_j)$, and let $\lambda_{j,1},\dots,\lambda_{j,d}>0$ be the eigenvalues of $\Sigma_j$. Choose
\[
0<t_0<\min_{1\le j\le m}\frac{1}{4\lambda_{\max}(\Sigma_j)}.
\]
Writing $Z_j=\mu_j+\Sigma_j^{1/2}W$ with $W\sim \mathrm{N}(0,I_d)$ and using $\left\lVert a+b \right\rVert^2\le 2\left\lVert a \right\rVert^2+2\left\lVert b \right\rVert^2$, we obtain
\[
\left\lVert Z_j \right\rVert^2\le 2\left\lVert \mu_j \right\rVert^2+2\left\lVert \Sigma_j^{1/2}W \right\rVert^2.
\]
Hence
\[
\int_{\mathbb{R}^d} e^{t_0\left\lVert x \right\rVert^2}\varphi(x;\mu_j,\Sigma_j)\,\mathrm{d}x
\le e^{2t_0\left\lVert \mu_j \right\rVert^2}
\int_{\mathbb{R}^d} e^{2t_0\left\lVert \Sigma_j^{1/2}w \right\rVert^2}\varphi(w;0,I_d)\,\mathrm{d}w.
\]
Let $U_j$ be an orthogonal matrix such that
\[
\Sigma_j = U_j \operatorname{diag}(\lambda_{j,1},\dots,\lambda_{j,d})U_j^\top.
\]
Since $\Sigma_j^{1/2}=U_j\operatorname{diag}(\lambda_{j,1}^{1/2},\dots,\lambda_{j,d}^{1/2})U_j^\top$ and $\varphi(w;0,I_d)$ is orthogonally invariant, the change of variables $u=U_j^\top w$ gives
\[
\int_{\mathbb{R}^d} e^{2t_0\left\lVert \Sigma_j^{1/2}w \right\rVert^2}\varphi(w;0,I_d)\,\mathrm{d}w
=
\int_{\mathbb{R}^d}
e^{2t_0\sum_{\ell=1}^d \lambda_{j,\ell}u_\ell^2}
\prod_{\ell=1}^d (2\pi)^{-1/2}e^{-u_\ell^2/2}\,\mathrm{d}u.
\]
Now the integrand factorizes into a product of one-dimensional functions, so Fubini's theorem yields
\[
\int_{\mathbb{R}^d} e^{2t_0\left\lVert \Sigma_j^{1/2}w \right\rVert^2}\varphi(w;0,I_d)\,\mathrm{d}w
=
\prod_{\ell=1}^d \int_{\mathbb{R}} (2\pi)^{-1/2} e^{-(1-4t_0\lambda_{j,\ell})u^2/2}\,\mathrm{d}u
=
\prod_{\ell=1}^d (1-4t_0\lambda_{j,\ell})^{-1/2}<\infty,
\]
because $4t_0\lambda_{j,\ell}<1$ for every $j$ and $\ell$. Therefore each Gaussian component has a finite quadratic exponential moment at $t_0$, and so does the finite GMM $g$.
\end{proof}

\begin{theorem}[Universal necessity of the second moment]
\label{thm:necessity}
Let $f$ be any density on $\mathbb{R}^d$. If there exists $g\in\mathcal M$ such that $\mathrm{KL}(f\|g)<\infty$, then
\[
\int_{\mathbb{R}^d}\left\lVert x \right\rVert^2 f(x)\,\mathrm{d}x<\infty.
\]
In particular, if
\[
\inf_{g\in\mathcal M}\mathrm{KL}(f\|g)=0,
\]
then the same second-moment conclusion holds.
\end{theorem}

\begin{proof}
Fix $g\in\mathcal M$ with $\mathrm{KL}(f\|g)<\infty$. By Lemma~\ref{lem:gauss-exp}, there exists $t_0>0$ such that
\[
\int e^{t_0\left\lVert x \right\rVert^2}g(x)\,\mathrm{d}x<\infty.
\]
Apply Lemma~\ref{lem:variational} with $\phi(x)=t_0\left\lVert x \right\rVert^2$ to obtain
\[
t_0\int \left\lVert x \right\rVert^2 f(x)\,\mathrm{d}x
\le \mathrm{KL}(f\|g)+\log\!\left(\int e^{t_0\left\lVert x \right\rVert^2}g(x)\,\mathrm{d}x\right)<\infty.
\]
Hence $\int \left\lVert x \right\rVert^2 f(x)\,\mathrm{d}x<\infty$.

If $\inf_{g\in\mathcal M}\mathrm{KL}(f\|g)=0$, then in particular there exists $g\in\mathcal M$ with $\mathrm{KL}(f\|g)<1$, and the first part applies.
\end{proof}

\subsection{The meta theorem}

The preceding results separate the universal ingredients from the class-specific verification. The following theorem packages that separation into a single abstract criterion.

\begin{theorem}
\label{thm:metaiff}
Let $\mathcal F$ be a class of densities on $\mathbb{R}^d$ with the following property that for every $f\in\mathcal F$ satisfying $\int \left\lVert x \right\rVert^2 f(x)\,\mathrm{d}x<\infty$, there exists a sequence $(g_m)_{m\ge1}\subset\mathcal M$ such that
\begin{enumerate}[label=\textnormal{(\alph*)}]
    \item $f/g_m\to1$ for $F$-almost every $x$;
    \item the family $\{L_m\}_{m\ge1}$, where $L_m=\log_+(f/g_m)$, is uniformly integrable under $F$.
\end{enumerate}
Then, for every $f\in\mathcal F$,
\[
\inf_{g\in\mathcal M}\mathrm{KL}(f\|g)=0
\quad\Longleftrightarrow\quad
\int_{\mathbb{R}^d}\left\lVert x \right\rVert^2f(x)\,\mathrm{d}x<\infty.
\]
\end{theorem}

\begin{proof}
If $\int \left\lVert x \right\rVert^2 f(x)\,\mathrm{d}x<\infty$, then by the defining property of $\mathcal F$ there exists a sequence $(g_m)_{m\ge1}\subset\mathcal M$ satisfying the hypotheses of Proposition~\ref{prop:UIKL}. Hence $\mathrm{KL}(f\|g_m)\to0$, which implies
\[
\inf_{g\in\mathcal M}\mathrm{KL}(f\|g)=0.
\]
Conversely, if $\inf_{g\in\mathcal M}\mathrm{KL}(f\|g)=0$, then Theorem~\ref{thm:necessity} gives $\int \left\lVert x \right\rVert^2 f(x)\,\mathrm{d}x<\infty$.
\end{proof}

\begin{remark}
Theorem~\ref{thm:metaiff} identifies the broadest class obtainable by this method once one is willing to define the class through the existence of uniformly-integrable likelihood-ratio approximants. In practice, however, such an existence-based description is too broad to be useful, so the substantive task is to identify concrete subclasses on which the pointwise convergence and uniform-integrability requirements can be verified by transparent structural hypotheses.
\end{remark}

\section{Classes on which equivalence holds}\label{sec:classes}

We now verify the abstract criterion on two concrete families of target densities. The first is a positivity class controlled by a finite log-moment condition. The second is a support-aware class modeled on \citet{NoretsPelenis2012} and its countable-scale extension.

\subsection{A finite log-moment approach}

\begin{definition}[Finite log-moment class]
\label{def:Fent}
Let $\mathcal F_{\mathrm{Ent}}$ denote the class of all continuous strictly positive probability densities $f$ on $\mathbb{R}^d$ such that
\[
\int_{\mathbb{R}^d} f(x)\log_+ f(x)\,\mathrm{d}x<\infty.
\]
\end{definition}

\begin{theorem}[Equivalence on the finite log-moment class]
\label{thm:entclass-UI}
For every $f\in\mathcal F_{\mathrm{Ent}}$,
\[
\inf_{g\in\mathcal M}\mathrm{KL}(f\|g)=0
\quad\Longleftrightarrow\quad
\int_{\mathbb{R}^d}\left\lVert x \right\rVert^2f(x)\,\mathrm{d}x<\infty.
\]
\end{theorem}

\begin{proof}
Assume first that
\[
\int_{\mathbb{R}^d}\left\lVert x \right\rVert^2f(x)\,\mathrm{d}x<\infty.
\]
For $R>0$, define the tail quantities
\[
\alpha_R=\int_{B_{R}^c} f(x)\,\mathrm{d}x,
\qquad
\beta_R=\int_{B_{R}^c} f(x)\log_+ f(x)\,\mathrm{d}x,
\qquad
\mu_{2,R}=\int_{B_{R}^c}\left\lVert x \right\rVert^2 f(x)\,\mathrm{d}x.
\]
Since $f\log_+f\in \mathcal L^1(\mathbb{R}^d)$ and $\left\lVert \cdot \right\rVert^2 f\in \mathcal L^1(\mathbb{R}^d)$, we have
\[
\alpha_R\to0,
\qquad
\beta_R\to0,
\qquad
\mu_{2,R}\to0
\qquad (R\to\infty).
\]
Moreover, because $\alpha_R\to0$,
\[
\alpha_R\log\!\left(\frac1{\alpha_R}\right)\to0.
\]
Also $\alpha_R\log R\to0$: if $s_2=\int \left\lVert x \right\rVert^2f(x)\,\mathrm{d}x<\infty$, then Markov's inequality gives $\alpha_R\le s_2/R^2$, hence
\[
0\le \alpha_R\log R\le s_2\frac{\log R}{R^2}\to0.
\]

Choose a sequence $(R_m)_{m\ge1}$ such that $R_m\uparrow\infty$, $R_m\ge1$ for every $m$, and, for every $m\ge1$,
\begin{equation}
\label{eq:Rmchoice}
\begin{aligned}
\alpha_{R_m}&<\tfrac14,
&\beta_{R_m}&<2^{-m-4},
&\alpha_{R_m}\log\!\left(\frac1{\alpha_{R_m}}\right)&<2^{-m-4},\\
\alpha_{R_m}\frac d2\log(2\pi R_m^2)&<2^{-m-4},
&\frac{\mu_{2,R_m}}{2R_m^2}&<2^{-m-4}.
\end{aligned}
\end{equation}
This is possible by the convergences above.

For each $m$, continuity and strict positivity imply
\[
m_m=\inf_{x\in B_{R_m}} f(x)>0.
\]
Apply Lemma~\ref{lem:compactapprox} with radius $R_m$ and tolerance
\[
\eta_m=2^{-m-3}m_m
\]
to obtain a finite GMM $h_m\in\mathcal M$ such that
\begin{equation}
\label{eq:hmclose}
\sup_{x\in B_{R_m}}|h_m(x)-f(x)|<\eta_m.
\end{equation}
In particular,
\begin{equation}
\label{eq:hmlower}
h_m(x)\ge f(x)-\eta_m\ge m_m-2^{-m-3}m_m\ge \frac{m_m}{2}
\qquad \text{for all }x\in B_{R_m}.
\end{equation}
Now define
\[
\phi_m(x)=\varphi(x;0,R_m^2 I_d),
\qquad
g_m(x)=(1-\alpha_{R_m})h_m(x)+\alpha_{R_m}\phi_m(x).
\]
Then each $g_m$ is a finite GMM, hence $g_m\in\mathcal M$.

Fix $x\in\mathbb{R}^d$. For all sufficiently large $m$, $x\in B_{R_m}$, so by \eqref{eq:hmclose},
\[
|h_m(x)-f(x)|\le \eta_m=2^{-m-3}m_m\le 2^{-m-3}f(x)\to0.
\]
Also $\alpha_{R_m}\to0$, and
\[
\phi_m(x)=(2\pi R_m^2)^{-d/2}\exp\!\left(-\frac{\left\lVert x \right\rVert^2}{2R_m^2}\right)\to0.
\]
Therefore
\[
|g_m(x)-f(x)|
\le (1-\alpha_{R_m})|h_m(x)-f(x)| + \alpha_{R_m}f(x)+\alpha_{R_m}\phi_m(x)\to0.
\]
Since $f(x)>0$ for every $x$, it follows that $f(x)/g_m(x)\to1$ pointwise on $\mathbb{R}^d$.

We next estimate $L_m=\log_+(f/g_m)$ in $\mathcal L^1(F)$. Since $g_m\ge (1-\alpha_{R_m})h_m$, we have for $x\in B_{R_m}$,
\[
\log\!\left(\frac{f(x)}{g_m(x)}\right)
\le \log\!\left(\frac{f(x)}{(1-\alpha_{R_m})h_m(x)}\right)
=\log\!\left(\frac{f(x)}{h_m(x)}\right)+\log\!\left(\frac1{1-\alpha_{R_m}}\right).
\]
Hence
\[
L_m(x)\le \left|\log\!\left(\frac{f(x)}{h_m(x)}\right)\right|+\log\!\left(\frac1{1-\alpha_{R_m}}\right).
\]
On $B_{R_m}$, both $f$ and $h_m$ are at least $m_m/2$ by \eqref{eq:hmlower}; therefore
\[
\left|\log\!\left(\frac{f(x)}{h_m(x)}\right)\right|
\le \frac{|f(x)-h_m(x)|}{m_m/2}
\le \frac{2\eta_m}{m_m}=2^{-m-2}.
\]
Using $\log(1/(1-u))\le 2u$ for $u\in[0,1/2]$ and \eqref{eq:Rmchoice}, we conclude that on $B_{R_m}$,
\begin{equation}
\label{eq:compactL1}
\int_{B_{R_m}}L_m\,\mathrm{d}F \le 2^{-m-2}+2\alpha_{R_m}.
\end{equation}

Since $R_m\ge1$ and $x\in B_{R_m}^c$, we have
\[
\phi_m(x)
=(2\pi R_m^2)^{-d/2}\exp\!\left(-\frac{\left\lVert x \right\rVert^2}{2R_m^2}\right)
\le (2\pi R_m^2)^{-d/2}e^{-1/2}
\le (2\pi)^{-d/2}e^{-1/2}<1.
\]
Because $\alpha_{R_m}<1/4$, it follows that
\[
\log\!\left(\frac{1}{\alpha_{R_m}\phi_m(x)}\right)\ge0.
\]
Since $g_m\ge \alpha_{R_m}\phi_m$ on all of $\mathbb{R}^d$, we therefore obtain on $B_{R_m}^c$,
\[
L_m(x)
\le \left(\log f(x)+\log\!\left(\frac1{\alpha_{R_m}\phi_m(x)}\right)\right)_+
\le \log_+ f(x)+\log\!\left(\frac1{\alpha_{R_m}}\right)-\log\phi_m(x).
\]
Because
\[
-\log\phi_m(x)=\frac d2\log(2\pi R_m^2)+\frac{\left\lVert x \right\rVert^2}{2R_m^2},
\]
integration over $B_{R_m}^c$ yields
\begin{equation}
\label{eq:tailL1}
\int_{B_{R_m}^c}L_m\,\mathrm{d}F
\le \beta_{R_m}+\alpha_{R_m}\log\!\left(\frac1{\alpha_{R_m}}\right)+\alpha_{R_m}\frac d2\log(2\pi R_m^2)+\frac{\mu_{2,R_m}}{2R_m^2}.
\end{equation}
By \eqref{eq:Rmchoice}, the right-hand side is at most $2^{-m-2}$.

Combining \eqref{eq:compactL1} and \eqref{eq:tailL1}, and using again \eqref{eq:Rmchoice}, we obtain
\[
\int L_m\,\mathrm{d}F \le 2^{-m-2}+2\alpha_{R_m}+2^{-m-2}\to0.
\]
In particular, $L_m\to0$ in $\mathcal L^1(F)$. To verify uniform integrability directly, let $\varepsilon>0$. Choose $m_0$ so large that
\[
\int L_m\,\mathrm{d}F<\varepsilon
\qquad\text{for all }m\ge m_0.
\]
For the finitely many indices $1\le m<m_0$, choose $M<\infty$ such that
\[
\int_{\{L_m>M\}}L_m\,\mathrm{d}F<\varepsilon
\qquad\text{for all }1\le m<m_0.
\]
Then, for $m\ge m_0$,
\[
\int_{\{L_m>M\}}L_m\,\mathrm{d}F\le \int L_m\,\mathrm{d}F<\varepsilon,
\]
and for $1\le m<m_0$ the same bound holds by construction. Hence
\[
\sup_{m\ge1}\int_{\{L_m>M\}}L_m\,\mathrm{d}F<\varepsilon,
\]
so the family $\{L_m\}_{m\ge1}$ is uniformly integrable. Proposition~\ref{prop:UIKL} therefore gives
\[
\mathrm{KL}(f\|g_m)\to0.
\]
This proves that finite second-moment existence implies KL-approximability and yields the approximating sequence. The implication
\[
\inf_{g\in\mathcal M}\mathrm{KL}(f\|g)=0
\quad\Longrightarrow\quad
\int_{\mathbb{R}^d}\left\lVert x \right\rVert^2f(x)\,\mathrm{d}x<\infty
\]
follows from Theorem~\ref{thm:necessity}.
\end{proof}

\subsection{Orlicz-dominated classes and concrete consequences}

We present the Orlicz formulation first because the bounded and $\mathcal L^p$ subclasses are immediate consequences of it.

\begin{definition}
Let $\Phi:[0,\infty)\to[0,\infty)$ be measurable. We say that $\Phi$ dominates the positive entropy if there exist constants $A,B\ge0$ such that
\[
t\log_+t\le A\Phi(t)+Bt
\qquad\text{for all }t\ge0.
\]
Let $\mathcal F_{\Phi}$ denote the class of continuous strictly positive densities $f$ such that
\[
\int_{\mathbb{R}^d}\Phi(f(x))\,\mathrm{d}x<\infty.
\]
\end{definition}

\begin{corollary}
\label{cor:orlicz}
Suppose $\Phi$ dominates the positive entropy. Then, for every $f\in\mathcal F_{\Phi}$,
\[
\inf_{g\in\mathcal M}\mathrm{KL}(f\|g)=0
\quad\Longleftrightarrow\quad
\int_{\mathbb{R}^d}\left\lVert x \right\rVert^2f(x)\,\mathrm{d}x<\infty.
\]
\end{corollary}

\begin{proof}
If $f\in\mathcal F_{\Phi}$, then
\[
\int f\log_+f\,\mathrm{d}x
\le A\int \Phi(f(x))\,\mathrm{d}x + B\int f(x)\,\mathrm{d}x <\infty.
\]
Therefore $f\in\mathcal F_{\mathrm{Ent}}$, and Theorem~\ref{thm:entclass-UI} applies.
\end{proof}

\begin{definition}
Let $\mathcal F_{\infty}$ be the class of continuous strictly positive densities $f$ on $\mathbb{R}^d$ such that $\|f\|_{\infty}<\infty$.
\end{definition}

\begin{corollary}
\label{cor:bounded}
For every $f\in\mathcal F_{\infty}$,
\[
\inf_{g\in\mathcal M}\mathrm{KL}(f\|g)=0
\quad\Longleftrightarrow\quad
\int_{\mathbb{R}^d}\left\lVert x \right\rVert^2f(x)\,\mathrm{d}x<\infty.
\]
\end{corollary}

\begin{proof}
Since $f$ is bounded,
\[
f(x)\log_+f(x)\le f(x)\log_+\|f\|_\infty
\]
for all $x$. Hence $\int f\log_+f\,\mathrm{d}x<\infty$, so $f\in\mathcal F_{\mathrm{Ent}}$, and Theorem~\ref{thm:entclass-UI} applies.
\end{proof}

\begin{definition}
For $p>1$, let $\mathcal F_p$ be the class of continuous strictly positive densities $f$ on $\mathbb{R}^d$ such that $f\in \mathcal L^p(\mathbb{R}^d)$.
\end{definition}

\begin{lemma}
\label{lem:plogp}
Let $p>1$. Then for every $t\ge0$,
\[
t\log_+ t\le \frac{t^p}{p-1}.
\]
\end{lemma}

\begin{proof}
If $0\le t\le1$, then $\log_+ t=0$ and the claim is trivial. If $t\ge1$, the elementary inequality $\log t\le t^{p-1}/(p-1)$ yields
\[
t\log t\le \frac{t^p}{p-1}.
\]
\end{proof}

\begin{corollary} 
\label{cor:Lp}
Let $p>1$. For every $f\in\mathcal F_p$,
\[
\inf_{g\in\mathcal M}\mathrm{KL}(f\|g)=0
\quad\Longleftrightarrow\quad
\int_{\mathbb{R}^d}\left\lVert x \right\rVert^2f(x)\,\mathrm{d}x<\infty.
\]
\end{corollary}

\begin{proof}
By Lemma~\ref{lem:plogp},
\[
\int f\log_+f\,\mathrm{d}x\le \frac{1}{p-1}\int f(x)^p\,\mathrm{d}x<\infty.
\]
Thus $f\in\mathcal F_{\mathrm{Ent}}$, and Theorem~\ref{thm:entclass-UI} applies.
\end{proof}

\begin{remark}\label{rem:rateclasses}
The target-density classes used in \citet[Theorem~1]{KruijerRousseauvdV2010} and \citet[Definition~2.1]{MaugisRabusseauMichel2013} are contained in $\mathcal F_{\infty}$, and hence in $\mathcal F_{\mathrm{Ent}}$. In both papers, $\log f$ is assumed locally $\beta$-H\"older, which implies continuity of $f$, and strict positivity is assumed explicitly. Their tail conditions then bound $f$ by deterministic envelopes that decay to $0$ in the tails: polynomial or exponential in \citet{KruijerRousseauvdV2010}, and Gaussian in \citet{MaugisRabusseauMichel2013}. Therefore each such density is bounded in the tails, while continuity yields boundedness on every compact set. Hence these classes fall inside $\mathcal F_{\infty}$, so Corollary~\ref{cor:bounded} applies to them.
\end{remark}

\subsection{The fixed-scale support-aware class}

We now turn to the support-aware framework. The external fixed-scale result behind this part of the manuscript is \citet[Theorem~3]{NoretsPelenis2012}. The next definition records the fixed-scale support-regularity part of that theorem. The second-moment assumption is kept separate because it is precisely the universal moment condition whose necessity and sufficiency we track here. The model-side requirement in \citet{NoretsPelenis2012} that arbitrarily small covariance scales be available is absorbed into our ambient class $\mathcal M$.

\begin{definition}[The fixed-scale support-aware class]
\label{def:Sfix}
Let $f$ be a density on $\mathbb{R}^d$. For each $m\ge1$ let
\[
\operatorname{supp}(f)=A^{(m)}_0\sqcup A^{(m)}_1\sqcup\cdots\sqcup A^{(m)}_{m}
\]
be a partition of $\operatorname{supp}(f)$ in which $A^{(m)}_1,\dots,A^{(m)}_{m}$ are adjacent cubes with common side length $h_m>0$ and $A^{(m)}_0$ is the remainder.

We say that $f\in\mathcal F_{\mathrm{FSSA}}$ if:
\begin{enumerate}[label=\textnormal{(\roman*)}]
    \item $f$ is continuous on $\operatorname{supp}(f)$ except on a set of $F$-measure zero;
    \item there exists $r>0$ and, for every $y\in \operatorname{supp}(f)$, a comparison cube $C(r,y)\subset \mathbb{R}^d$ of side length $r$ containing $y$ such that the map
    \[
    D_r(y)=\log\!\left(\frac{f(y)}{\inf_{z\in C(r,y)}f(z)}\right)
    \]
    is measurable on $\operatorname{supp}(f)$ and
    \[
    \int_{\operatorname{supp}(f)} D_r(y)\,\mathrm{d}F(y)<\infty;
    \]
    \item there exists $M\in\mathbb N$ such that for every $m\ge M$:
    \begin{enumerate}[label=\textnormal{(\alph*)}]
        \item if $y\in A^{(m)}_0$, then $C(r,y)\cap A^{(m)}_0$ contains a cube of side $r/2$ with a vertex at $y$;
        \item if $y\in \operatorname{supp}(f)\setminus A^{(m)}_0$, then $C(r,y)\cap \left(\operatorname{supp}(f)\setminus A^{(m)}_0\right)$ contains a cube of side $r/2$ with a vertex at $y$.
    \end{enumerate}
\end{enumerate}
\end{definition}

\begin{theorem}[Equivalence on the fixed-scale support-aware class]
\label{thm:Sfixiff}
If $f\in\mathcal F_{\mathrm{FSSA}}$, then
\[
\inf_{g\in\mathcal M}\mathrm{KL}(f\|g)=0
\quad\Longleftrightarrow\quad
\int_{\mathbb{R}^d}\left\lVert x \right\rVert^2f(x)\,\mathrm{d}x<\infty.
\]
\end{theorem}

\begin{proof}
Assume first that
\[
\int_{\mathbb{R}^d}\left\lVert x \right\rVert^2f(x)\,\mathrm{d}x<\infty.
\]
By Definition~\ref{def:Sfix}, the density $f$ satisfies the continuity, fixed-scale oscillation, and support-geometry hypotheses in \citet[Theorem~3]{NoretsPelenis2012}, with $Y=\operatorname{supp}(f)$. The remaining model-side requirement in \citet[Theorem~3]{NoretsPelenis2012} is that arbitrarily small covariance scales be available. This is automatic for our ambient class $\mathcal M$, since finite Gaussian mixtures with covariance matrices $\sigma^2 I_d$ belong to $\mathcal M$ for every $\sigma>0$. Therefore, for every $\varepsilon>0$ there exists $g\in\mathcal M$ such that
\[
\mathrm{KL}(f\|g)<\varepsilon,
\]
and hence
\[
\inf_{g\in\mathcal M}\mathrm{KL}(f\|g)=0.
\]
The converse implication is exactly the necessity statement from Theorem~\ref{thm:necessity}.
\end{proof}

\begin{remark}
The support-geometry condition is a local thickness requirement at the working scale $r$. It says that when the comparison cube $C(r,y)$ meets either the remainder region or the cubical part of the partition, the relevant side of the support still contains a subcube of side $r/2$ anchored at $y$. In particular, the condition rules out pieces of the support that, at scale $r$, become too thin to contain such an anchored subcube or that break locally into fragments separated by gaps inside the comparison cube. This is exactly what allows the lower-envelope argument of \citet{NoretsPelenis2012} to work.
\end{remark}

\subsection{A countable-scale support-aware class}

We now allow the fixed local scale to vary across a measurable partition of the support. The definition is built by requiring the normalized restriction of the density to each partition piece to satisfy the fixed-scale support-aware condition, with summability across the pieces strong enough to recover the same KL equivalence.

\begin{definition}[The countable-scale support-aware class]
\label{def:Scnt}
Let $f$ be a density on $\mathbb{R}^d$. We say that $f\in\mathcal F_{\mathrm{CSSA}}$ if there exist:
\begin{enumerate}[label=\textnormal{(\roman*)}]
    \item a measurable partition
    \[
    \operatorname{supp}(f)=\bigsqcup_{\ell\ge1} Y_\ell,
    \qquad
    p_\ell=F(Y_\ell)\ge0;
    \]
    \item for each $\ell$ with $p_\ell>0$, a scale $r_\ell>0$, a family of comparison cubes $C_\ell(y)\subset \mathbb{R}^d$ of side length $r_\ell$ containing $y$ for $y\in Y_\ell$, and, for every $m\ge1$, a partition
    \[
    Y_\ell=A^{(m)}_{\ell,0}\sqcup A^{(m)}_{\ell,1}\sqcup\cdots\sqcup A^{(m)}_{\ell,m}
    \]
    in which $A^{(m)}_{\ell,1},\dots,A^{(m)}_{\ell,m}$ are adjacent cubes with common side length $h_{\ell,m}>0$ and $A^{(m)}_{\ell,0}$ is the remainder,
\end{enumerate}
such that, for every $\ell$ with $p_\ell>0$, the normalized restriction
\[
f_\ell(x)=\frac{f(x)\mathbf{1}_{Y_\ell}(x)}{p_\ell}
\]
satisfies Definition~\ref{def:Sfix} on $Y_\ell$ with witness scale $r_\ell$, comparison cubes $C_\ell(y)$, and partitions $A^{(m)}_{\ell,0},\dots,A^{(m)}_{\ell,m}$, and such that, writing $F_\ell$ for the probability measure with density $f_\ell$ and
\[
D_\ell(y)=\log\!\left(\frac{f_\ell(y)}{\inf_{z\in C_\ell(y)} f_\ell(z)}\right),
\qquad y\in Y_\ell,
\]
one has
\[
\sum_{\ell:p_\ell>0} p_\ell\int_{Y_\ell} D_\ell(y)\,\mathrm{d}F_\ell(y)<\infty
\]
and
\[
\sum_{\ell:p_\ell>0} p_\ell\log_+\!\left(\frac{1}{r_\ell}\right)<\infty.
\]
Zero-mass pieces are allowed and may be ignored.
\end{definition}

\begin{remark}
Taking a single piece $Y_1=\operatorname{supp}(f)$ and $p_1=1$, together with the original witnesses from Definition~\ref{def:Sfix}, shows immediately that
\[
\mathcal F_{\mathrm{FSSA}}\subset \mathcal F_{\mathrm{CSSA}}.
\]
\end{remark}

\begin{remark}
Relative to $\mathcal F_{\mathrm{FSSA}}$, the new flexibility is that the admissible local scale may vary from piece to piece. The two summability conditions ensure that the piecewise oscillation integrals and the admissible scales do not deteriorate too quickly on pieces carrying non-negligible mass.
\end{remark}

We first record the fixed-scale entropy bound needed in the proof of Proposition~\ref{prop:Scnt-entropy}.

\begin{lemma}[Entropy bound on the fixed-scale class]
\label{lem:Sfix-entropy}
Let $q\in\mathcal F_{\mathrm{FSSA}}$, let $Q$ denote the probability measure with density $q$, and let $r>0$, $D_r$, $C(r,y)$, $A^{(m)}_0,\dots,A^{(m)}_m$, and $M$ be witnesses for Definition~\ref{def:Sfix}. Then
\[
q\log_+ q\in \mathcal L^1(\mathbb{R}^d)
\]
and
\[
\int_{\mathbb{R}^d} q(y)\log_+ q(y)\,\mathrm{d}y
\le
\int_{\operatorname{supp}(q)} D_r(y)\,\mathrm{d}Q(y)+d\log_+\!\left(\frac{2}{r}\right).
\]
\end{lemma}

\begin{proof}
Fix $m=M$. If $y\in A^{(M)}_0$, then Definition~\ref{def:Sfix}\textnormal{(iii)(a)} yields a cube $Q_y\subset C(r,y)\cap A^{(M)}_0$ of side $r/2$ with a vertex at $y$. If $y\in \operatorname{supp}(q)\setminus A^{(M)}_0$, then Definition~\ref{def:Sfix}\textnormal{(iii)(b)} yields a cube $Q_y\subset C(r,y)\cap \left(\operatorname{supp}(q)\setminus A^{(M)}_0\right)$ of side $r/2$ with a vertex at $y$. In either case $Q_y\subset \operatorname{supp}(q)$, so
\[
1\ge \int_{Q_y} q(z)\,\mathrm{d}z
\ge \left(\frac{r}{2}\right)^d\inf_{z\in Q_y} q(z)
\ge \left(\frac{r}{2}\right)^d\inf_{z\in C(r,y)} q(z).
\]
Hence
\[
q(y)\le \frac{q(y)}{\inf_{z\in C(r,y)}q(z)}\left(\frac{2}{r}\right)^d.
\]
Since $q(y)/\inf_{z\in C(r,y)}q(z)\ge1$, it follows that
\[
\log_+ q(y)\le D_r(y)+d\log_+\!\left(\frac{2}{r}\right)
\]
for $Q$-almost every $y\in \operatorname{supp}(q)$. Integrating gives the stated bound.
\end{proof}

\begin{proposition}
\label{prop:Scnt-entropy}
If $f\in\mathcal F_{\mathrm{CSSA}}$, then $f\log_+f\in \mathcal L^1(\mathbb{R}^d)$. More precisely,
\[
\int_{\mathbb{R}^d} f(y)\log_+ f(y)\,\mathrm{d}y
\le
\sum_{\ell:p_\ell>0} p_\ell\int_{Y_\ell} D_\ell(y)\,\mathrm{d}F_\ell(y)
+d\sum_{\ell:p_\ell>0} p_\ell\log_+\!\left(\frac{2}{r_\ell}\right).
\]
In particular,
\[
\int_{\mathbb{R}^d} f(y)\log_+ f(y)\,\mathrm{d}y<\infty.
\]
\end{proposition}

\begin{proof}
Discard any zero-mass pieces and relabel the remaining pieces so that $p_\ell>0$ for every index used below. On $Y_\ell$ one has $f=p_\ell f_\ell$, and since $0<p_\ell\le1$,
\[
\log_+\!\left(p_\ell f_\ell(y)\right)\le \log_+ f_\ell(y).
\]
Therefore
\[
\int_{Y_\ell} f(y)\log_+ f(y)\,\mathrm{d}y
\le p_\ell\int_{Y_\ell} f_\ell(y)\log_+ f_\ell(y)\,\mathrm{d}y.
\]
Applying Lemma~\ref{lem:Sfix-entropy} to the fixed-scale witness of $f_\ell$ gives
\[
\int_{Y_\ell} f_\ell(y)\log_+ f_\ell(y)\,\mathrm{d}y
\le
\int_{Y_\ell} D_\ell(y)\,\mathrm{d}F_\ell(y)+d\log_+\!\left(\frac{2}{r_\ell}\right).
\]
Multiplying by $p_\ell$ and summing over $\ell$ yields
\[
\int_{\mathbb{R}^d} f(y)\log_+ f(y)\,\mathrm{d}y
\le
\sum_{\ell\ge1} p_\ell\int_{Y_\ell} D_\ell(y)\,\mathrm{d}F_\ell(y)
+d\sum_{\ell\ge1} p_\ell\log_+\!\left(\frac{2}{r_\ell}\right).
\]
Finally,
\[
\log_+\!\left(\frac{2}{r_\ell}\right)\le \log 2+\log_+\!\left(\frac{1}{r_\ell}\right),
\]
and $\sum_{\ell\ge1}p_\ell=1$, so the right-hand side is finite by Definition~\ref{def:Scnt}.
\end{proof}

\begin{lemma}[Mixture convexity of the KL divergence]
\label{lem:KL-convex}
Let $a_1,\dots,a_L\ge0$ with $\sum_{i=1}^L a_i=1$, let $q_i$ be densities on $\mathbb{R}^d$, and let $r_i$ be strictly positive densities on $\mathbb{R}^d$. Then
\[
\mathrm{KL}\!\left(\sum_{i=1}^L a_i q_i \middle\| \sum_{i=1}^L a_i r_i\right)
\le \sum_{i=1}^L a_i\mathrm{KL}(q_i\|r_i).
\]
\end{lemma}

\begin{proof}
If $\sum_{i=1}^L a_i\mathrm{KL}(q_i\|r_i)=\infty$, there is nothing to prove. Discard any indices with $a_i=0$, since they contribute zero to both sides of the desired inequality. Thus we may assume that $a_i>0$ for every $i$ under consideration. Applying Lemma~\ref{lem:logsum} pointwise with
\[
\alpha_i(x)=a_i q_i(x),
\qquad
\beta_i(x)=a_i r_i(x)
\qquad (1\le i\le L),
\]
we obtain for almost every $x\in\mathbb{R}^d$ that
\[
\left(\sum_{i=1}^L a_i q_i(x)\right)\log\!\left(\frac{\sum_{i=1}^L a_i q_i(x)}{\sum_{i=1}^L a_i r_i(x)}\right)
\le \sum_{i=1}^L a_i q_i(x)\log\!\left(\frac{a_i q_i(x)}{a_i r_i(x)}\right).
\]
Since each $a_i>0$, the factor $a_i$ cancels inside the logarithm, and therefore
\[
\left(\sum_{i=1}^L a_i q_i(x)\right)\log\!\left(\frac{\sum_{i=1}^L a_i q_i(x)}{\sum_{i=1}^L a_i r_i(x)}\right)
\le \sum_{i=1}^L a_i q_i(x)\log\!\left(\frac{q_i(x)}{r_i(x)}\right).
\]
Integrating over $\mathbb{R}^d$ gives the stated inequality.
\end{proof}

\begin{theorem}
\label{thm:Scntiff}
If $f\in\mathcal F_{\mathrm{CSSA}}$, then
\[
\inf_{g\in\mathcal M}\mathrm{KL}(f\|g)=0
\quad\Longleftrightarrow\quad
\int_{\mathbb{R}^d}\left\lVert x \right\rVert^2f(x)\,\mathrm{d}x<\infty.
\]
\end{theorem}

\begin{proof}
Discard any zero-mass pieces and relabel the remaining pieces so that $p_\ell>0$ for every index used below. This does not affect the defining conditions.

Assume first that
\[
\int_{\mathbb{R}^d}\left\lVert x \right\rVert^2f(x)\,\mathrm{d}x<\infty.
\]
For each $\ell\ge1$, Definition~\ref{def:Scnt} gives $f_\ell\in\mathcal F_{\mathrm{FSSA}}$. Also,
\[
\int_{\mathbb{R}^d}\left\lVert x \right\rVert^2 f_\ell(x)\,\mathrm{d}x
=\frac{1}{p_\ell}\int_{Y_\ell}\left\lVert x \right\rVert^2f(x)\,\mathrm{d}x<\infty.
\]
Hence Theorem~\ref{thm:Sfixiff} implies that for every $\ell$ and every $\varepsilon>0$ there exists $g_\ell\in\mathcal M$ such that
\[
\mathrm{KL}(f_\ell\|g_\ell)<\varepsilon 2^{-\ell-1}.
\]
Fix the standard Gaussian density
\[
h(x)=\varphi(x;0,I_d).
\]
For $L\ge1$, define
\[
B_L=\bigcup_{\ell>L}Y_\ell,
\qquad
p_{>L}=F(B_L)=\sum_{\ell>L}p_\ell,
\qquad
\nu_L(x)=\frac{f(x)\mathbf{1}_{B_L}(x)}{p_{>L}}
\]
when $p_{>L}>0$; if $p_{>L}=0$, let $\nu_L=h$. Also define the finite GMM
\[
G_L(x)=\sum_{\ell=1}^L p_\ell g_\ell(x)+p_{>L}h(x).
\]
Because
\[
f(x)=\sum_{\ell=1}^L p_\ell f_\ell(x)+p_{>L}\nu_L(x),
\]
Lemma~\ref{lem:KL-convex} yields
\begin{equation}
\label{eq:Scnt-KLsplit}
\mathrm{KL}(f\|G_L)
\le \sum_{\ell=1}^L p_\ell\mathrm{KL}(f_\ell\|g_\ell)+p_{>L}\mathrm{KL}(\nu_L\|h).
\end{equation}
By construction,
\begin{equation}
\label{eq:Scnt-firstterm}
\sum_{\ell=1}^L p_\ell\mathrm{KL}(f_\ell\|g_\ell)\le \varepsilon/2.
\end{equation}

If $p_{>L}=0$, then $\eqref{eq:Scnt-KLsplit}$ and \eqref{eq:Scnt-firstterm} already give $\mathrm{KL}(f\|G_L)\le \varepsilon/2$. Thus it remains only to treat the case $p_{>L}>0$.

Since $-\log h(x)=\frac d2\log(2\pi)+\frac12\left\lVert x \right\rVert^2$, we have
\begin{align}
\label{eq:Scnt-tailcomp}
p_{>L}\mathrm{KL}(\nu_L\|h)
&=\int_{B_L}\log\!\left(\frac{f(x)}{p_{>L}h(x)}\right)\,\mathrm{d}F(x)\notag\\
&=\int_{B_L}\log\!\left(\frac{f(x)}{h(x)}\right)\,\mathrm{d}F(x)+p_{>L}\log\!\left(\frac1{p_{>L}}\right)\notag\\
&\le \int_{B_L}\log_+ f(x)\,\mathrm{d}F(x)
+\frac d2\log(2\pi)\,p_{>L}
+\frac12\int_{B_L}\left\lVert x \right\rVert^2\,\mathrm{d}F(x)
+p_{>L}\log\!\left(\frac1{p_{>L}}\right).
\end{align}
By Proposition~\ref{prop:Scnt-entropy}, $f\log_+f\in \mathcal L^1(\mathbb{R}^d)$. Since $B_L\downarrow\varnothing$ as $L\to\infty$, absolute continuity of the integral gives
\[
\int_{B_L}\log_+ f(x)\,\mathrm{d}F(x)\to0,
\qquad
\int_{B_L}\left\lVert x \right\rVert^2\,\mathrm{d}F(x)\to0,
\qquad
p_{>L}\to0.
\]
Moreover, $p_{>L}\log(1/p_{>L})\to0$ as $L\to\infty$. Therefore the right-hand side of \eqref{eq:Scnt-tailcomp} converges to $0$, and so there exists $L=L(\varepsilon)$ such that
\begin{equation}
\label{eq:Scnt-tailterm}
p_{>L}\mathrm{KL}(\nu_L\|h)<\varepsilon/2.
\end{equation}
Combining \eqref{eq:Scnt-KLsplit}, \eqref{eq:Scnt-firstterm}, and \eqref{eq:Scnt-tailterm}, we obtain
\[
\mathrm{KL}(f\|G_L)<\varepsilon.
\]
Since $\varepsilon>0$ was arbitrary and $G_L\in\mathcal M$, this proves
\[
\inf_{g\in\mathcal M}\mathrm{KL}(f\|g)=0.
\]
The converse implication is exactly the necessity statement from Theorem~\ref{thm:necessity}, namely
\[
\inf_{g\in\mathcal M}\mathrm{KL}(f\|g)=0
\Longrightarrow
\int_{\mathbb R^d} \|x\|^2 f(x)\,d\lambda(x)<\infty.
\]
\end{proof}

\section{Counterexamples}\label{sec:counterexamples}

The next three results clarify the relation between $\mathcal F_{\mathrm{FSSA}}$, $\mathcal F_{\mathrm{CSSA}}$, and $\mathcal F_{\mathrm{Ent}}$. Proposition~\ref{prop:Scnt-strict} proves that $\mathcal F_{\mathrm{CSSA}}$ strictly extends $\mathcal F_{\mathrm{FSSA}}$ and reaches densities outside $\mathcal F_{\mathrm{Ent}}$. Proposition~\ref{prop:Ent-not-Scnt} then shows that the finite log-moment class is not contained in the countable-scale support-aware class. Together, these two results show that $\mathcal F_{\mathrm{Ent}}$ and $\mathcal F_{\mathrm{CSSA}}$ do not contain one another. Proposition~\ref{prop:routes-not-exhaustive} finally shows that even the union of the finite log-moment and countable-scale support-aware classes is not exhaustive.

\begin{proposition}
\label{prop:Scnt-strict}
Define
\[
Z_\star=\sum_{n=1}^\infty n^{-6},
\qquad
J_n=[n^2,n^2+n^{-6}],
\qquad
f_\star(x)=\frac{1}{Z_\star}\sum_{n=1}^\infty \mathbf{1}_{J_n}(x),
\qquad x\in\mathbb{R}.
\]
Then $f_\star$ is a probability density on $\mathbb{R}$ with finite second moment, and
\[
f_\star\in\mathcal F_{\mathrm{CSSA}},
\qquad
f_\star\notin\mathcal F_{\mathrm{FSSA}},
\qquad
f_\star\notin\mathcal F_{\mathrm{Ent}}.
\]
In particular,
\[
\mathcal F_{\mathrm{FSSA}}\subsetneq \mathcal F_{\mathrm{CSSA}}.
\]
\end{proposition}

\begin{proof}
The sets $J_n$ are pairwise disjoint and have total mass
\[
\int_{\mathbb{R}} f_\star(x)\,\mathrm{d}x
=\frac{1}{Z_\star}\sum_{n=1}^\infty |J_n|
=\frac{1}{Z_\star}\sum_{n=1}^\infty n^{-6}=1,
\]
so $f_\star$ is a density. Moreover,
\[
\int_{\mathbb{R}} x^2 f_\star(x)\,\mathrm{d}x
\le \frac{1}{Z_\star}\sum_{n=1}^\infty (n^2+n^{-6})^2 n^{-6}
\le \frac{4}{Z_\star}\sum_{n=1}^\infty n^{-2}<\infty,
\]
so $f_\star$ has finite second moment. Let $F_\star$ denote the probability measure with density $f_\star$.

We first verify that $f_\star\in \mathcal F_{\mathrm{CSSA}}$. Set $Y_n=J_n$. Then the sets $(Y_n)_{n\ge1}$ form a measurable partition of $\mathrm{supp}(f_\star)$ and
\[
p_n=F_\star(Y_n)=\frac{n^{-6}}{Z_\star}.
\]
For each $n$, the normalized restriction is
\[
f_n(x)=\frac{f_\star(x)\mathbf{1}_{Y_n}(x)}{p_n}
=\frac{\mathbf{1}_{J_n}(x)}{|J_n|}.
\]
Let $F_n$ denote the probability measure with density $f_n$. Define $r_n=n^{-6}/4$. For $x\in Y_n$, define the interval $C_n(x)$ by
\[
C_n(x)=
\begin{cases}
[x,x+r_n], & x\in [n^2,n^2+r_n),\\[4pt]
[x-r_n/2,x+r_n/2], & x\in [n^2+r_n,n^2+n^{-6}-r_n],\\[4pt]
[x-r_n,x], & x\in (n^2+n^{-6}-r_n,n^2+n^{-6}].
\end{cases}
\]
Then $C_n(x)\subset Y_n$ for every $x\in Y_n$, each $C_n(x)$ has side length $r_n$, and $f_n$ is constant on $Y_n$. Therefore
\[
D_n(x)=\log\!\left(\frac{f_n(x)}{\inf_{z\in C_n(x)}f_n(z)}\right)=0
\qquad\text{for every }x\in Y_n.
\]
For the support-geometry clause, partition $Y_n$ into $m$ adjacent intervals of equal length $n^{-6}/m$:
\[
A^{(m)}_{n,j}=\left[n^2+(j-1)\frac{n^{-6}}{m},\,n^2+j\frac{n^{-6}}{m}\right)
\quad (1\le j<m),
\qquad
A^{(m)}_{n,m}=\left[n^2+(m-1)\frac{n^{-6}}{m},\,n^2+n^{-6}\right],
\]
and let $A^{(m)}_{n,0}=\varnothing$. Since $C_n(x)\subset Y_n$ and, by construction, every $C_n(x)\cap Y_n$ contains an interval of length $r_n/2$ with a vertex at $x$, the density $f_n$ belongs to $\mathcal F_{\mathrm{FSSA}}$ with witness scale $r_n$.

Moreover,
\[
\sum_{n=1}^\infty p_n\int_{Y_n} D_n(x)\,\mathrm{d}F_n(x)=0
\]
and
\[
\sum_{n=1}^\infty p_n\log_+\!\left(\frac1{r_n}\right)
=\frac{1}{Z_\star}\sum_{n=1}^\infty n^{-6}\log(4n^6)<\infty.
\]
Hence $f_\star\in\mathcal F_{\mathrm{CSSA}}$.

We next show that $f_\star\notin\mathcal F_{\mathrm{FSSA}}$. Suppose to the contrary that $f_\star$ belongs to that class with some fixed local scale $r>0$. Choose $n$ so large that $n^{-6}<r$. If $x\in (n^2,n^2+n^{-6})$, then every interval $C(r,x)$ of length $r$ containing $x$ must extend outside $J_n$, because $|J_n|=n^{-6}<r$. Since $f_\star$ vanishes on the open gap adjacent to $J_n$, such an interval contains a point $z$ with $f_\star(z)=0$. Hence
\[
\inf_{z\in C(r,x)} f_\star(z)=0
\qquad\text{for every }x\in (n^2,n^2+n^{-6}),
\]
and therefore the defining oscillation integrand for $\mathcal F_{\mathrm{FSSA}}$ is infinite on a set of positive $F_\star$-measure. This contradiction shows that $f_\star\notin\mathcal F_{\mathrm{FSSA}}$.

Finally, $f_\star\notin\mathcal F_{\mathrm{Ent}}$ because $\mathcal F_{\mathrm{Ent}}$ consists of continuous strictly positive densities on all of $\mathbb{R}^d$, whereas $f_\star$ vanishes on the complement of $\bigcup_n J_n$ and is discontinuous at the endpoints of the intervals $J_n$. Since $f_\star\in\mathcal F_{\mathrm{CSSA}}$ and has finite second moment, Theorem~\ref{thm:Scntiff} gives
\[
\inf_{g\in\mathcal M}\mathrm{KL}(f_\star\|g)=0.
\]
This completes the proof.
\end{proof}

\begin{proposition}
\label{prop:Ent-not-Scnt}
There exists a density $f_\natural\in\mathcal F_{\mathrm{Ent}}$ such that
\[
\int_{\mathbb{R}} x^2 f_\natural(x)\,\mathrm{d}x<\infty,
\qquad
f_\natural\notin\mathcal F_{\mathrm{CSSA}}.
\]
\end{proposition}

\begin{proof}
For $n\ge2$, let
\[
w_n=e^{-n^4},
\qquad
N_n=\left\lceil \frac{e^{n^4}}{32n^4}\right\rceil.
\]
For $k=0,\dots,N_n-1$, define
\[
c_{n,k}=n+(4k+2)w_n,
\qquad
\psi_{n,k}(x)=\max\!\left\{1-\frac{|x-c_{n,k}|}{w_n},0\right\}.
\]
Thus $\psi_{n,k}$ is the usual tent function of height $1$, centered at $c_{n,k}$, with support
\[
S_{n,k}=\left[n+(4k+1)w_n,\,n+(4k+3)w_n\right],
\]
so $|S_{n,k}|=2w_n$. Let
\[
h_n(x)=\sum_{k=0}^{N_n-1}\psi_{n,k}(x),
\qquad
I_n=[n,\,n+4N_nw_n],
\qquad
\nu(x)=e^{-x^4}+\sum_{n=2}^{\infty} h_n(x),
\qquad
Z_\natural=\int_{\mathbb{R}} \nu(x)\,\mathrm{d}x,
\]
and define
\[
f_\natural(x)=\frac{\nu(x)}{Z_\natural}.
\]
Since
\[
4N_nw_n\le \frac{1}{8n^4}+4e^{-n^4}<\frac14
\qquad (n\ge2),
\]
one has $I_n\subset [n,n+1/4]$ for every $n\ge2$. In particular, the intervals $(I_n)_{n\ge2}$ are pairwise disjoint. Hence, at each $x\in\mathbb{R}$, at most one of the functions $h_n(x)$ is nonzero, so the series defining $\nu$ is locally finite. Therefore $\nu$ is continuous. Also, $\nu(x)\ge e^{-x^4}>0$ for every $x$, and
\[
0<\nu(x)\le e^{-x^4}+1\le 2
\qquad \text{for all }x\in\mathbb{R}.
\]
Moreover, each tent function has area $w_n$, so
\[
\int_{\mathbb{R}} h_n(x)\,\mathrm{d}x=N_nw_n\le \frac{1}{32n^4}+e^{-n^4}.
\]
Consequently,
\[
0<Z_\natural\le \int_{\mathbb{R}} e^{-x^4}\,\mathrm{d}x+\sum_{n=2}^{\infty}\left(\frac{1}{32n^4}+e^{-n^4}\right)<\infty.
\]
Thus $f_\natural$ is a continuous strictly positive probability density on $\mathbb{R}$. Since $f_\natural\le 2/Z_\natural$, we have
\[
f_\natural(x)\log_+ f_\natural(x)
\le f_\natural(x)\log_+\!\left(\frac{2}{Z_\natural}\right)
\qquad \text{for all }x\in\mathbb{R},
\]
and the right-hand side is integrable. Hence $f_\natural\in\mathcal F_{\mathrm{Ent}}$.

We next verify the second moment. Since $x\mapsto e^{-x^4}$ has finite second moment and $I_n\subset[n,n+1/4]$, one has $x^2\le (n+1)^2$ on $I_n$. Therefore
\begin{align*}
\int_{\mathbb{R}} x^2\nu(x)\,\mathrm{d}x
&\le \int_{\mathbb{R}} x^2 e^{-x^4}\,\mathrm{d}x
+\sum_{n=2}^{\infty}(n+1)^2\int_{\mathbb{R}} h_n(x)\,\mathrm{d}x\\
&\le \int_{\mathbb{R}} x^2 e^{-x^4}\,\mathrm{d}x
+\sum_{n=2}^{\infty}(n+1)^2\left(\frac{1}{32n^4}+e^{-n^4}\right)<\infty.
\end{align*}
Thus $f_\natural$ has finite second moment. Let $F_\natural$ denote the probability measure with density $f_\natural$.

For each $n\ge2$ and $k=0,\dots,N_n-1$, define the central interval
\[
J_{n,k}=\left[c_{n,k}-\frac{w_n}{2},\,c_{n,k}+\frac{w_n}{2}\right],
\]
and let
\[
H_n=\bigcup_{k=0}^{N_n-1} J_{n,k}.
\]
If $x\in J_{n,k}$, then $|x-c_{n,k}|\le w_n/2$, and hence $\psi_{n,k}(x)\ge 1/2$. Thus $h_n(x)\ge1/2$ on $H_n$. Since $H_n\subset I_n$ and the intervals $(I_n)_{n\ge2}$ are pairwise disjoint, the sets $(H_n)_{n\ge2}$ are pairwise disjoint as well. Also,
\[
|H_n|=N_nw_n\ge \frac{1}{32n^4}.
\]
Therefore,
\[
F_\natural(H_n)=\frac{1}{Z_\natural}\int_{H_n}\nu(x)\,\mathrm{d}x
\ge \frac{1}{2Z_\natural}|H_n|
\ge \frac{1}{64Z_\natural}n^{-4}.
\]

We now prove that $f_\natural\notin\mathcal F_{\mathrm{CSSA}}$. Suppose to the contrary that $f_\natural\in\mathcal F_{\mathrm{CSSA}}$, and let $(Y_\ell)_{\ell\ge1}$, $(p_\ell)_{\ell\ge1}$, $(r_\ell)_{\ell\ge1}$, and the piecewise fixed-scale witnesses from Definition~\ref{def:Scnt} be given. For each $\ell$ with $p_\ell>0$, let
\[
f_\ell(x)=\frac{f_\natural(x)\mathbf 1_{Y_\ell}(x)}{p_\ell},
\]
let $F_\ell$ denote the probability measure with density $f_\ell$, and let
\[
D_\ell(y)=\log\!\left(\frac{f_\ell(y)}{\inf_{z\in C_\ell(y)}f_\ell(z)}\right),
\qquad y\in Y_\ell.
\]
Define piecewise functions on $\mathbb{R}$ by
\[
D(y)=D_\ell(y),
\qquad
R(y)=\log_+\!\left(\frac{1}{r_\ell}\right)
\qquad \text{for } y\in Y_\ell \text{ and } p_\ell>0,
\]
and set $D(y)=R(y)=0$ otherwise. Then
\[
\int_{\mathbb{R}} D(y)\,\mathrm{d}F_\natural(y)
=\sum_{\ell:p_\ell>0} p_\ell\int_{Y_\ell} D_\ell(y)\,\mathrm{d}F_\ell(y)<\infty
\]
and
\[
\int_{\mathbb{R}} R(y)\,\mathrm{d}F_\natural(y)
=\sum_{\ell:p_\ell>0} p_\ell\log_+\!\left(\frac{1}{r_\ell}\right)<\infty.
\]
We claim that, for every $n\ge2$ and every $y\in H_n$,
\[
D(y)+R(y)\ge n^4-\log 2.
\]
Indeed, fix $n\ge2$ and $y\in H_n$, and let $k$ be the unique index such that $y\in J_{n,k}$. Let $\ell$ be the unique index with $y\in Y_\ell$.

If $r_\ell\le 2w_n$, then
\[
R(y)=\log_+\!\left(\frac{1}{r_\ell}\right)
\ge \log\!\left(\frac{1}{2w_n}\right)
= n^4-\log 2.
\]

Suppose now that $r_\ell>2w_n$. Since $C_\ell(y)$ is an interval of length $r_\ell$ containing $y$ and $|S_{n,k}|=2w_n<r_\ell$, the interval $C_\ell(y)$ cannot be contained in $S_{n,k}$. Every support interval $S_{n,k}$ is separated from the rest of $\bigcup_{j=0}^{N_n-1}S_{n,j}$ by an open gap on each side lying inside $I_n$ (for the first and last supports, the outer adjacent gap has length $w_n$; otherwise the adjacent gaps have length $2w_n$). Because $C_\ell(y)$ is connected and contains a point of $S_{n,k}$ but is not contained in $S_{n,k}$, it must contain a point
\[
z\in I_n\setminus \bigcup_{j=0}^{N_n-1} S_{n,j}.
\]
If $z\notin Y_\ell$, then $f_\ell(z)=0$, and hence $D(y)=\infty$. It therefore remains only to consider the case $z\in Y_\ell$. In that case, $h_n(z)=0$ and $h_m(z)=0$ for every $m\neq n$, so
\[
\nu(z)=e^{-z^4}\le e^{-n^4},
\]
because $z\in I_n\subset[n,n+1/4]$. On the other hand, since $y\in H_n$, we have $\nu(y)\ge1/2$. Therefore,
\[
D(y)
\ge \log\!\left(\frac{f_\ell(y)}{f_\ell(z)}\right)
=\log\!\left(\frac{f_\natural(y)}{f_\natural(z)}\right)
=\log\!\left(\frac{\nu(y)}{\nu(z)}\right)
\ge \log\!\left(\frac{1/2}{e^{-n^4}}\right)
= n^4-\log 2.
\]
This proves the claim.

Integrating the claim over $H_n$ and summing over $n\ge2$, we obtain
\[
\int_{\mathbb{R}} (D(y)+R(y))\,\mathrm{d}F_\natural(y)
\ge \sum_{n=2}^{\infty}\int_{H_n}(D(y)+R(y))\,\mathrm{d}F_\natural(y)
\ge \sum_{n=2}^{\infty}(n^4-\log 2)F_\natural(H_n).
\]
Using the lower bound on $F_\natural(H_n)$, this yields
\[
\int_{\mathbb{R}} (D(y)+R(y))\,\mathrm{d}F_\natural(y)
\ge \frac{1}{64Z_\natural}\sum_{n=2}^{\infty}\left(1-\frac{\log 2}{n^4}\right)=\infty,
\]
which contradicts the finiteness of $\int D\,\mathrm{d}F_\natural$ and $\int R\,\mathrm{d}F_\natural$. Therefore $f_\natural\notin\mathcal F_{\mathrm{CSSA}}$. 
This completes the proof.
\end{proof}

\begin{definition}[Fat Cantor set]
\label{def:fatcantor}
A fat Cantor set is a compact nowhere dense subset of $[0,1]$ with strictly positive Lebesgue measure. A concrete construction is given in \citet[Chapter~8, pp.~87--88]{GelbaumOlmsted1964}.
\end{definition}

\begin{proposition}
\label{prop:routes-not-exhaustive}
Let $A\subset[0,1]$ be a fat Cantor set and define
\[
c=\lambda(A)^{-1},
\qquad
f^{\dagger}(x)=c\,\mathbf{1}_A(x),
\qquad x\in\mathbb{R}.
\]
Then $f^{\dagger}$ is a probability density on $\mathbb{R}$ with finite second moment, and
\[
\inf_{g\in\mathcal M}\mathrm{KL}(f^{\dagger}\|g)=0,
\qquad
f^{\dagger}\notin\mathcal F_{\mathrm{Ent}},
\qquad
f^{\dagger}\notin\mathcal F_{\mathrm{CSSA}}.
\]
\end{proposition}

\begin{proof}
Since $A\subset[0,1]$ has strictly positive finite Lebesgue measure,
\[
\int_{\mathbb{R}} f^{\dagger}(x)\,\mathrm{d}x
=c\,\lambda(A)=1.
\]
Thus $f^{\dagger}$ is a probability density. Its second moment is finite because $f^{\dagger}$ is supported on $[0,1]$. Let $F^{\dagger}$ denote the probability measure with density $f^{\dagger}$.

Also $f^{\dagger}\notin\mathcal F_{\mathrm{Ent}}$, because $\mathcal F_{\mathrm{Ent}}$ consists of continuous strictly positive densities on all of $\mathbb{R}^d$, whereas $f^{\dagger}$ vanishes on $\mathbb{R}\setminus A$ and is discontinuous at every point of $A=\partial A$.

We next show that $f^{\dagger}\notin\mathcal F_{\mathrm{CSSA}}$. Suppose to the contrary that
\[
A=\bigsqcup_{\ell\ge1} Y_\ell
\]
is a partition witnessing membership in $\mathcal F_{\mathrm{CSSA}}$. Since
\[
1=F^{\dagger}(A)=\sum_{\ell\ge1} F^{\dagger}(Y_\ell),
\]
there exists $\ell$ such that $F^{\dagger}(Y_\ell)>0$. Let
\[
f_\ell(x)=\frac{f^{\dagger}(x)\mathbf{1}_{Y_\ell}(x)}{F^{\dagger}(Y_\ell)}.
\]
By Definition~\ref{def:Scnt}, the density $f_\ell$ belongs to $\mathcal F_{\mathrm{FSSA}}$ with some scale $r_\ell>0$ and comparison cube $C_\ell(y)$ for every $y\in Y_\ell$. Moreover, the defining oscillation integral of $f_\ell$ is finite, so the set
\[
E_\ell=\left\{y\in Y_\ell: \log\!\left(\frac{f_\ell(y)}{\inf_{z\in C_\ell(y)}f_\ell(z)}\right)<\infty\right\}
\]
has full $F_\ell$-measure and is therefore nonempty. Pick $y\in E_\ell$. Then
\[
\inf_{z\in C_\ell(y)} f_\ell(z)>0.
\]
Since $f_\ell$ vanishes on $\mathbb{R}\setminus Y_\ell$, this forces $C_\ell(y)\subset Y_\ell\subset A$. But $C_\ell(y)$ is a nondegenerate interval of length $r_\ell$, so it has nonempty interior, whereas the fat Cantor set $A$ has empty interior because it is nowhere dense. This contradiction shows that $f^{\dagger}\notin\mathcal F_{\mathrm{CSSA}}$.

It remains to prove KL approximability. Since $A$ is compact and Lebesgue measurable, outer regularity yields bounded open sets $O_m$ such that
\[
A\subset O_m\subset (-1,2),
\qquad
\lambda(O_m\setminus A)<m^{-1}.
\]
Because $\mathbb{R}$ is normal, Urysohn's lemma yields functions $u_m\in C_c(\mathbb{R};[0,1])$ such that
\[
\mathbf{1}_A\le u_m\le \mathbf{1}_{O_m}.
\]
Define
\[
I_m=\int_{\mathbb{R}} u_m(x)\,\mathrm{d}x,
\qquad
c_m=I_m^{-1},
\qquad
s_m(x)=c_m u_m(x).
\]
Then each $s_m$ is a continuous compactly supported density on $\mathbb{R}$, and on $A$ one has $u_m=1$, hence
\[
s_m(x)=c_m
\qquad\text{for every }x\in A.
\]
Moreover,
\[
\lambda(A)\le I_m\le \lambda(O_m)\le \lambda(A)+m^{-1},
\]
so $c_m\to c$.

Apply Lemma~\ref{lem:compactapprox} with $R=2$ and a sequence of tolerances $(\eta_m)_{m\ge1}$ satisfying $\eta_m\downarrow0$ and $\eta_m<c_m/2$ for every $m$, obtaining $g_m\in\mathcal M$ with
\[
\sup_{x\in B_2} |g_m(x)-s_m(x)|<\eta_m.
\]
Since $A\subset[0,1]\subset B_2$, it follows that on $A$,
\[
|g_m(x)-c_m|<\eta_m,
\qquad
\text{hence}
\qquad
g_m(x)\ge c_m-\eta_m>\frac{c_m}{2}.
\]
Therefore,
\begin{align*}
\mathrm{KL}(f^{\dagger}\|g_m)
&=\int_A c\log\!\left(\frac{c}{g_m(x)}\right)\,\mathrm{d}x\\
&=\int_A c\log\!\left(\frac{c}{c_m}\right)\,\mathrm{d}x+\int_A c\log\!\left(\frac{c_m}{g_m(x)}\right)\,\mathrm{d}x\\
&\le \left|\log\!\left(\frac{c}{c_m}\right)\right|+\sup_{x\in A}\left|\log\!\left(\frac{c_m}{g_m(x)}\right)\right|.
\end{align*}
On $A$, both $c_m$ and $g_m(x)$ lie in $[c_m/2,3c_m/2]$, so the mean value theorem for $\log$ gives
\[
\left|\log\!\left(\frac{c_m}{g_m(x)}\right)\right|
\le \frac{|c_m-g_m(x)|}{c_m/2}
\le \frac{2\eta_m}{c_m}
\qquad (x\in A).
\]
Consequently,
\[
\mathrm{KL}(f^{\dagger}\|g_m)
\le \left|\log\!\left(\frac{c}{c_m}\right)\right|+\frac{2\eta_m}{c_m}\to0.
\]
Hence $\inf_{g\in\mathcal M}\mathrm{KL}(f^{\dagger}\|g)=0$, as claimed.
\end{proof}

\begin{remark}
The two sufficient classes treated in this manuscript capture the structural hypotheses that arise most naturally in qualitative GMM approximation: either the target is globally positive with enough logarithmic integrability to control the finite log-ratios, or it has a support with enough local geometric regularity to support the lower-envelope method. Proposition~\ref{prop:routes-not-exhaustive} shows that this does not exhaust all KL-approximable densities. The density constructed there lies outside both sufficient classes because its support is deliberately singular: it vanishes off its support and is discontinuous on the support boundary, so it fails the finite log-moment class, while the fat Cantor support contains no nondegenerate interval, so the local cube-thickness mechanism behind the support-aware classes cannot apply.
\end{remark}

\section{Conclusion}\label{sec:conclusion}

The main outcome of this manuscript is that the second moment is necessary for KL approximation by finite GMMs. The necessity statement is unconditional and depends only on the elementary fact that every finite GMM has a quadratic exponential moment. The sufficient direction is more delicate and is not attached to any single canonical structural class. Instead, the meta theorem separates the universal probabilistic part of the argument from the class-specific verification.

In this manuscript we verified that mechanism via two approaches. The first approach is global: continuity, strict positivity, and a finite log-moment condition are enough to force the finite log-ratio family to be uniformly integrable for a suitable compact approximation scheme. This approach automatically covers bounded, $\mathcal L^p$ with $p>1$, and Orlicz-dominated subclasses, and it also contains the logarithmic H\"older classes used in \citet{KruijerRousseauvdV2010} and \citet{MaugisRabusseauMichel2013}; see Remark~\ref{rem:rateclasses}. The second approach is geometric and support-aware: it includes the fixed-scale class from \citet{NoretsPelenis2012} and the countable-scale extension introduced here. Proposition~\ref{prop:Scnt-strict} shows that the countable-scale class is strictly larger than this fixed-scale class and reaches densities with support boundaries and zeros that are not covered by the present finite log-moment class. Proposition~\ref{prop:Ent-not-Scnt} shows in the opposite direction that the finite log-moment class is not contained in the countable-scale support-aware class, so the two sufficient classes do not contain one another. Proposition~\ref{prop:routes-not-exhaustive} shows, moreover, that their union is still not exhaustive.

Several directions remain open. One natural problem is to weaken the finite log-moment approach so as to admit zeros of the target density without importing the complicated support geometry of the support-aware framework. Another is to sharpen the present qualitative equivalence into quantitative rates under minimal moment and regularity assumptions, to extend upon the rate results for the logarithmic H\"older classes of \citet{KruijerRousseauvdV2010} and \citet{MaugisRabusseauMichel2013}. A third direction comes from conditional density estimation by covariate-dependent mixtures. \citet{NoretsPelenis2014} extend the Gaussian-mixture program to a conditional-distribution framework involving covariate-dependent mixing weights. This perspective is closely related to the mixture-of-experts approximation literature; see, for example, \citet{NguyenChamroukhiForbes2019} and \citet{NguyenHienTTChamroukhiMcLachlan2021}. It would be interesting to determine whether the necessity argument and the meta-sufficiency principle developed here admit an analogue in that setting, thereby yielding a sharper characterization of Gaussian mixture-of-experts approximation of conditional distributions.

\appendix

\section{Auxiliary results}\label{app:external}

\begin{theorem}[Vitali convergence theorem]
\label{thm:vitali}
Let $(X_n)_{n\ge1}$ be a sequence of nonnegative random variables on $(\Omega,\mathscr F,\mathbb P)$ such that $X_n\to0$ almost surely and $\{X_n:n\ge1\}$ is uniformly integrable. Then
\[
\int_\Omega X_n\,\mathrm d\mathbb P\to0.
\]
\end{theorem}

\begin{lemma}[Log-sum inequality]
\label{lem:logsum}
If $a_i\ge 0$ and $b_i>0$ for $1\le i\le L$, then
\[
\left(\sum_{i=1}^L a_i\right)\log\!\left(\frac{\sum_{i=1}^L a_i}{\sum_{i=1}^L b_i}\right)
\le \sum_{i=1}^L a_i \log\!\left(\frac{a_i}{b_i}\right),
\]
with the usual convention $0\log 0=0$.
\end{lemma}

\begin{lemma}[Variational entropy inequality]
\label{lem:variational}
Let $f$ and $g$ be probability densities on $\mathbb{R}^d$, and let $F$ be the probability measure with density $f$. Let $\phi:\mathbb{R}^d\to[0,\infty]$ be measurable and assume
\[
Z_\phi=\int_{\mathbb{R}^d} e^{\phi(x)}g(x)\,\mathrm{d}x<\infty.
\]
Then
\[
\int \phi\,\mathrm{d}F \le \mathrm{KL}(f\|g)+\log Z_\phi.
\]
\end{lemma}

\begin{lemma}[Uniform approximation on a fixed ball]
\label{lem:compactapprox}
Let $f$ be a continuous probability density on $\mathbb{R}^d$. For every $R>0$ and every $\eta>0$, there exists a finite GMM $h\in\mathcal M$ such that
\[
\sup_{x\in B_R} |h(x)-f(x)|<\eta.
\]
\end{lemma}

\begin{remark}
Theorem~\ref{thm:vitali} is a standard Vitali convergence theorem; see \citet[Chapter~16]{Billingsley1995}. The log-sum inequality is classical and follows from Jensen's inequality for the convex function $u\mapsto u\log u$. Lemma~\ref{lem:variational} is a standard relative-entropy inequality; see \citet[Chapter~1]{DupuisEllis1997}. Lemma~\ref{lem:compactapprox} is the Gaussian specialization of Theorem~2.1\textnormal{(a)} of \citet{NguyenTTEtAl2022}.
\end{remark}

\end{document}